\documentclass{l4dc2021} 

\usepackage{times}
\usepackage{mathtools}
\usepackage[integrals]{wasysym}
\usepackage{caption}

\title[Chance-constrained quasi-convex optimization]{Chance-constrained quasi-convex optimization with application to data-driven switched systems control}

\author{%
\Name{Guillaume O. Berger} \Email{guillaume.berger@uclouvain.be}\\
\Name{Rapha\"el M. Jungers} \Email{raphael.jungers@uclouvain.be}\\
\Name{Zheming Wang} \Email{zheming.wang@uclouvain.be}\\
\addr Institute of Information and Communication Technologies, Electronics and Applied Mathematics, Department of Mathematical Engineering (ICTEAM/INMA) at UCLouvain, 1348 Louvain-la-Neuve, Belgium}

\renewcommand{\Re}{\mathbb{R}}

\newcommand{\NNb}{\mathbb{N}}

\newcommand{\SSb}{\mathbb{S}}

\newcommand{\diff}{\mathrm{d}}

\newcommand{\conv}{\mathrm{conv}}
\newcommand{\cone}{\mathrm{cone}}
\newcommand{\Prob}{\mathbb{P}}

\newcommand{\nn}{{n\times n}}
\newcommand{\lambdamin}{\lambda_{\mathrm{min}}}
\newcommand{\calA}{\mathcal{A}}

\newcommand{\calC}{\mathcal{C}}

\newcommand{\calP}{\mathcal{P}}

\newcommand{\calV}{\mathcal{V}}

\newcommand{\calX}{\mathcal{X}}

\newcommand{\Opt}{\mathrm{Opt}}
\newcommand{\Cost}{\mathrm{Cost}}
\newcommand{\Sub}{\mathrm{Sub}}
\newcommand{\Nor}{\mathrm{Nor}}

\newcommand{\ProbJSR}{\calP_{\mathrm{jsr}}}
\newcommand{\CostJSR}{\Cost_{\mathrm{jsr}}}
\newcommand{\OptJSR}{\Opt_{\mathrm{jsr}}}
\newcommand{\ViolJSR}{V_{\mathrm{jsr}}}

\newcommand{\omegab}{\bar{\omega}}
\newcommand{\uh}{\mathbf{1}}

\newcommand{\proofname}{Proof}
\newenvironment{proofbis}{\par\noindent{\bfseries\upshape \proofname\ }}{\jmlrQED}

\newtheorem{assumption}[theorem]{Assumption}

\makeatletter
\def\set@curr@file#1{\def\@curr@file{#1}} 
\makeatother

\begin{document}

\maketitle

\begin{abstract}
We study quasi-convex optimization problems, where only a subset of the constraints can be sampled, and yet one would like a probabilistic guarantee on the obtained solution with respect to the initial (unknown) optimization problem. Even though our results are partly applicable to general quasi-convex problems, in this work we introduce and study a particular subclass, which we call ``quasi-linear problems''.
We provide optimality conditions for these problems.
Thriving on this, we extend the approach of chance-constrained convex optimization to quasi-linear optimization problems.
Finally, we show that this approach is useful for the stability analysis of black-box switched linear systems, from a finite set of sampled trajectories.
It allows us to compute probabilistic upper bounds on the JSR of a large class of switched linear systems.
\end{abstract}



\section{Introduction}

Data-driven control has gained a lot of interest from the control community in recent years; see, e.g., \citet{duggirala2013verification,huang2014proofs,blanchini2016modelfree,kozarev2016case,balkan2017underminer,boczar2018finitedata}.
In many modern applications of control systems, one cannot rely on having a model of the system, but rather has to design a controller in a \emph{blackbox}, \emph{data-driven} fashion.
This is the case for instance for \emph{proprietary} systems; more usually, this happens because the system is too complex to be modeled, or because the obtained model is too complicated to be analyzed with classical control techniques.
In these situations, the control engineer can only rely on data --- which sometimes come in huge amounts ---, but make the problem of very different nature than the classical, model-based control problems.
Examples of such situations include self-driving cars, where the input to the controller is made of huge heterogeneous data (harvested from cameras, lidars, etc.); or smart grid applications, where the heterogeneous parts of the system (prosumers, smart buildings, etc.) are best described with data harvested from observing these parts than with a rigid, closed-form model \citep{aswani2012identifying,zhou2017quantitative}.

Data collected from a control system can be seen as \emph{samples} extracted from a large set of possible behaviors.
Controller design can then be approached by synthesizing controllers based on the sampled set of behaviors; the challenge is then to provide guarantees on the correctness of the controller for the whole behavior of the system.
In optimization, this approach is known as \emph{chance-constrained optimization}, which consists in sampling a \emph{subset} of the constraints of an optimization problem and solving the problem with these constraints only.
The solution obtained in this way will in general not satisfy \emph{all of the constraints} of the problem; however, probabilistic guarantees can be obtained on the measure of the set of constraints that are compatible with this solution \citep{calafiore2010random,margellos2014ontheroad,campi2018ageneral}.

The approach of chance-constrained optimization has already proved useful in several areas of control, like robust control design \citep{calafiore2006thescenario} or quantized control \citep{campi2018ageneral}.
Recently, it has been successfully applied to data-driven control problems, as a technique to bridge the gap between data and model-based control; see, e.g., applications in data-enabled predictive control \citep{van2015distributionally,coulson2020distributionally} and stability analysis of black-box dynamical systems \citep{kenanian2019data,wang2020adatadriven}.

In this work, we introduce a new class of optimization problems: \emph{quasi-linear problems}.
This class forms a subclass of quasi-convex optimization problems \citep[see, e.g.,][]{eppstein2005quasiconvex}.
We extend the results from \citet{calafiore2010random} for chance-constrained optimization of convex problems to quasi-linear optimization problems.
This is achieved by showing that for any such optimization problem there is a subset of constraints, called the \emph{essential set}, with bounded cardinality, that provides the same optimal solution as the original problem.
This result draws on an akin result for quasi-convex problems in \citet{eppstein2005quasiconvex}, and improves it in two ways: we get a better upper bound on the cardinality of essential sets, while removing the assumption that the constraints are ``continuously shrinking'' \citep{eppstein2005quasiconvex}.

We believe that chance-constrained quasi-linear optimization can find application in many areas of data-driven control.
For instance, by replacing LMIs with sampled linear inequalities, one could transform SDP problems in control \citep{boyd1994linear} into linear or quasi-linear programs, and use chance-constrained optimization to bridge the gap between the original and the sampled formulations.

As a proof of concept, we demonstrate here that the setting of chance-constrained quasi-linear optimization can be useful for the stability analysis of \emph{black-box switched linear systems}.
Switched Linear Systems are systems described by a finite set of linear modes among which the system can switch over time.
They constitute a paradigmatic class of hybrid and cyber-physical systems, and appear naturally in many engineering applications, or as abstractions of more complicated systems \citep{alur2009modeling,jadbabaie2003coordination}.
These systems turn out to be extremely challenging in terms of control and analysis, even for basic questions like stability or stabilizability.
In particular, the computation of the \emph{Joint Spectral Radius} (JSR), a measure of stability of switched linear systems, has been used as a benchmark for testing new approaches in complex systems \citep{blondel2005computationally,parrilo2008approximation,jungers2017acharacterization}.

Recently, the problem of JSR approximation was introduced for \emph{black-box} switched linear systems.
It is well known that bounds on the JSR of switched linear systems can be obtained from the resolution of adequate quasi-convex optimization problems built from the system \citep{jungers2017acharacterization}.
In \citet{kenanian2019data}, the authors extend this approach when the system is not known but only a few trajectories are observed, and apply chance-constrained optimization techniques to obtain probabilistic upper bounds and lower bounds on the JSR of the system.
In this work, we show that this approach fits in fact into the framework of chance-constrained quasi-linear optimization.
From this, probabilistic upper and lower bounds on the JSR of the system can be obtained straightforwardly, by applying the results introduced in this paper; the bounds obtained in that way are also better than the ones proposed in \citet{kenanian2019data}.

The paper is organized as follows.
In Section~\ref{sec-quasi-linear}, we introduce the class of quasi-linear optimization problems and discuss their properties.
In Section~\ref{sec-chance-constrained}, we state and prove the main theorem of this paper, which extends the results of chance-constrained \emph{convex} optimization to quasi-linear optimization problems.
Then, in Section~\ref{sec-JSR-black-box}, we apply the framework of chance-constrained quasi-linear optimization to the problem of stability analysis of black-box switched linear systems, and we show how this framework can be used to obtain probabilistic bounds on the JSR of the system.
Finally, in Section~\ref{sec-numerical-experiments}, we demonstrate the applicability of our results with several numerical examples.

{\itshape Notation.}
$\NNb$ denotes the set of nonnegative integers, and $\NNb_*$ the set of positive integers.
For a set of vectors $\calV\subseteq\Re^d$, $\conv(U)$ denotes the \emph{convex hull} of $U$, and $\cone(U)$ its \emph{conic hull}.
For a convex function $f:\Re^d\to\Re$ and $x\in\Re^d$, we let $\Sub_x(f)$ be the \emph{subdifferential} of $f$ at $x$, i.e., the set of vectors $g\in\Re^d$ such that $f(y)-f(x)\geq g^\top(y-x)$ for all $y\in\Re^d$; for a convex set $\calC\subseteq\Re^d$, we let $\Nor_x(\calC)$ be the \emph{normal cone} of $\calC$ at $x$, i.e., the set of vectors $g\in\Re^d$ such that $g^\top(y-x)\leq0$ for all $y\in\calC$.
If $\Delta$ is a set, $\omegab\coloneqq(\delta_1,\ldots,\delta_N)\in\Delta^N$ and $\delta_{N+1}\in\Delta$, we use $\omegab\Vert\delta_{N+1}$ to denote their \emph{concatenation}: $\omegab\Vert\delta_{N+1}\coloneqq(\delta_1,\ldots,\delta_{N+1})$; in Section~\ref{sec-chance-constrained}, for the sake of simplicity, we will slightly abuse the notation and write $\omega$ to denote the \emph{set} obtained from the elements of $\omegab\coloneqq(\delta_1,\ldots,\delta_N)$, i.e., $\omega=\{\delta_1,\ldots,\delta_N\}$.

\section{Quasi-linear optimization problems}\label{sec-quasi-linear}

In this section, we introduce a novel class of optimization problems, which are a particular case of quasi-convex problems. 
We particularize and improve some classical results of quasi-convex programming to this class.

Let $\calX$ be a compact convex subset of $\Re^d$, with nonempty interior and with $0\notin\calX$.
Let $\Delta$ be a set, and $\{a_\delta\}_{\delta\in\Delta}$ and $\{b_\delta\}_{\delta\in\Delta}$ be two collections --- indexed by $\delta\in\Delta$ --- of vectors in $\Re^d$ and such that $b_\delta^\top x>0$ for all $x\in\calX$ and $\delta\in\Delta$.
Consider the following optimization problem:
\begin{equation}\label{eq-optim}
\min\limits_{x\in\Re^d,\,\lambda\geq0}\; (\lambda,c(x)) \quad\text{s.t.}\quad x\in\calX, \quad\text{and}\quad a_\delta^\top x \leq \lambda b_\delta^\top x, \quad \forall\,\delta\in\Delta,
\end{equation}
where $c:\calX\to\Re$ is a strongly convex function.
The objective of \eqref{eq-optim} is to minimize $(\lambda,c(x))$ in the \emph{lexicographical order}%
\footnote{``First component first'': $(\lambda_1,c_1)<(\lambda_2,c_2)$ if $\lambda_1<\lambda_2$, or else $\lambda_1=\lambda_2$ and $c_1<c_2$.}%
, while respecting the constraints defined by $\Delta$ and $x\in\calX$.
See Figure~\ref{fig-quasi-linear} for an illustration.

\begin{figure}
\centering
\includegraphics[width=0.9\textwidth]{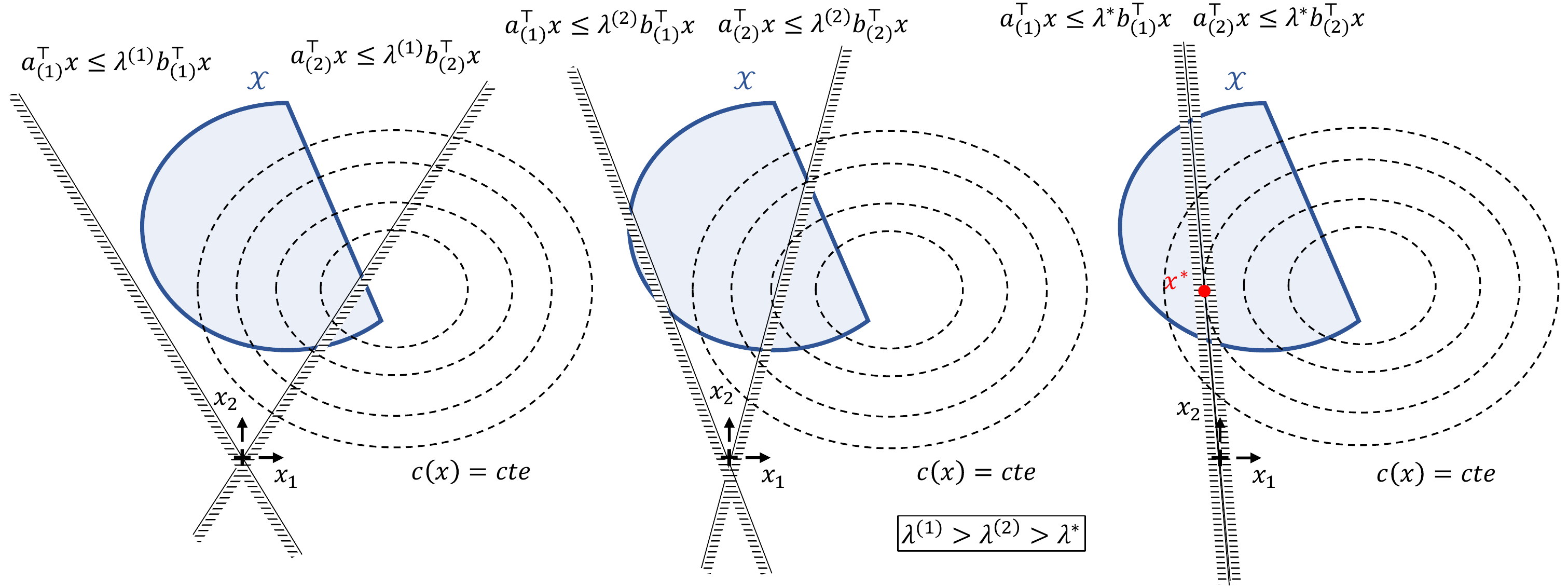}
\caption{Set of feasible points $x\in\Re^d$ of a quasi-linear optimization problem \eqref{eq-optim}, for three different values of $\lambda$.
The blue set $\calX$ represents the fixed constraints, while the two quasi-linear constraints are represented in black.
The dotted curves are level-curves of the secondary cost function $c(x)$.
The smallest $\lambda$ for which there is a feasible point $x$ is the optimal $\lambda$, denoted $\lambda^*$.
For this value of $\lambda$, the feasible point $x$ that minimizes $c$ is the optimal point $x$, denoted $x^*$.}
\label{fig-quasi-linear}
\end{figure}

Sometimes, it is not possible to solve \eqref{eq-optim} with all the constraints defined by $\Delta$, either because only a subset of these constraints are known (as it is the case for instance in data-driven control problems), or because the set $\Delta$ is so large (or even infinite) that it is algorithmically impracticable to enforce all of these constraints.
In these cases, for a finite set $\omega\subseteq\Delta$, we consider the following \emph{sampled} optimization problem:
\begin{equation}\label{eq-optim-sampled}
\calP(\omega):\quad\min\limits_{x\in\Re^d,\,\lambda\geq0}\; (\lambda,c(x)) \quad\text{s.t.}\quad x\in\calX, \quad\text{and}\quad a_\delta^\top x \leq \lambda b_\delta^\top x, \quad \forall\,\delta\in\omega.
\end{equation}
We let $\Opt(\omega)$ be the optimal solution%
\footnote{By the strong convexity of $c$, $\Opt(\omega)$ is unique.}
of $\calP(\omega)$ and we let $\Cost(\omega)$ be its optimal cost.
The constraints of $\calP(\omega)$ defined by $\omega$ will be called the \emph{sampled constraints}, while the constraint $x\in\calX$ is the \emph{common constraint}.

For a fixed value of $\lambda$, the sampled constraints of $\calP(\omega)$ are linear in $x$.
Therefore, we will say that $\calP(\omega)$ is a \emph{quasi-linear optimization problem}.
Note that $\calP(\omega)$ is a particular instance of quasi-convex optimization problems, as defined in \citet{eppstein2005quasiconvex}.
It is shown there that, under some technical assumption on the continuity of the constraints, the cardinality of any \emph{essential set} (see Definition~\ref{def-essential-set} below) of a quasi-convex problem is upper bounded by $d+1$, where $d$ is the dimension of $x$.
In this paper, we provide \emph{for quasi-linear problems} a better upper bound on the cardinality of their essential sets, and without the technical assumption of ``continuously shrinking'' constraints, present in \citet{eppstein2005quasiconvex}.

\begin{definition}\label{def-essential-set}
\citep[Definition~2.9]{calafiore2010random}
An \emph{essential set}\, for $\calP(\omega)$ is a set $\beta\subseteq\omega$, with minimal cardinality, satisfying $\Cost(\beta)=\Cost(\omega)$.
\end{definition}

\begin{theorem}\label{thm-essential-set-cardinality}
The cardinality of any essential set $\beta$ of\, $\calP(\omega)$ satisfies $\lvert\beta\rvert\leq d$.
\end{theorem}

To prove this theorem, we will need the following lemma.

\begin{lemma}\label{lem-condition-optimal}
\citep[Theorem~27.4]{rockafellar1970convex}
Let $f:\Re^d\to\Re$ be a convex function and $\calC\subseteq\Re^d$ a nonempty convex set.
Then, $x\in\calC$ is a mininizer%
\footnote{I.e., $f(x^*)=\inf_{x\in\calC}f(x)$.}
of $f$ over $\calC$ if and only if\, $0\in\Sub_x(f)+\Nor_x(\calC)$.
\end{lemma}

\begingroup
\renewcommand{\proofname}{Proof of Theorem~\ref{thm-essential-set-cardinality}}
\begin{proofbis}
Let $\beta$ be an essential set for $\calP(\omega)$ and let $(\lambda^*,x^*)=\Opt(\omega)$.
For each $\delta\in\omega$, let $h_\delta = a_\delta-\lambda^*b_\delta$.
Let $\gamma\subseteq\omega$ be the set of all $\delta\in\omega$ such that $h_\delta^\top x^*=0$.
We divide the proof in two cases.

\emph{Case~1:}
First, we consider the case when $\lambda^*=0$.
Assume that $x\in\calX$ is a \emph{support constraint}, meaning that the optimal cost of $\calP(\omega)$ without this constraint is strictly smaller than $\Cost(\omega)$.
Then, by the classical argument%
\footnote{Indeed, $\calP(\omega)$ with $\lambda$ fixed to zero is a convex optimization problem, and the cardinality of essential sets of feasible convex optimization problems is bounded by $d$ \citep[see, e.g.,][Theorem~3]{calafiore2006thescenario}.}%
, there is a set of at most $d$ constraints among those of $\calP(\omega)$ (i.e., among the constraints defined by $\omega$, and the constraint $x\in\calX$) such that the optimal solution of the problem with these constraints only is equal to $\Cost(\omega)$.
Because $x\in\calX$ is a support constraint, it must belong to this set of constraints.
Hence, there is a set $\beta'\subseteq\omega$, with $\lvert\beta'\rvert\leq d-1$, such that $\Cost(\beta')=\Cost(\omega)$.
This shows that $\lvert\beta\rvert\leq d-1$ when $x\in\calX$ is a support constraint.

Now, assume that $x\in\calX$ is not a support constraint, i.e., the optimal cost of $\calP(\omega)$ without this constraint is the same as $\Cost(\omega)$.
By Lemma~\ref{lem-condition-optimal}, it holds that $0\in\Sub_{x^*}(c)+\cone(\{h_\delta\}_{\delta\in\gamma})$.
Note that, by definition of $\gamma$, the vectors $\{h_\delta\}_{\delta\in\gamma}$ are all orthogonal to $x^*$, so that they belong to a $(d-1)$-dimensional subspace.
Hence, by Caratheodory theorem%
\footnote{See, e.g., \citet[Corollary~17.1.2]{rockafellar1970convex}.}%
, there is a set $\gamma'\subseteq\gamma$, with $\lvert\gamma'\rvert\leq d-1$, such that $0\in\Sub_{x^*}(c)+\cone(\{h_\delta\}_{\delta\in\gamma'})$.
By Lemma~\ref{lem-condition-optimal}, it thus follows that $\Cost(\gamma')=\Cost(\omega)$.
This shows that $\lvert\beta\rvert\leq d-1$ when $x\in\calX$ is not a support constraint; concluding the proof for the first case.

\emph{Case~2:}
Now, we consider the case when $\lambda^*>0$.
By Lemma~\ref{lem-condition-optimal} applied on $f(x)=\sup_{\delta\in\gamma}h_\delta^\top x$ and $\calC=\calX$, it follows that $0\in\conv(\{h_\delta\}_{\delta\in\gamma})+\Nor_{x^*}(\calX)$.
Let $\gamma'\subseteq\gamma$ be a nonempty subset with minimal cardinality such that $0\in\conv(\{h_\delta\}_{\delta\in\gamma'})+\Nor_{x^*}(\calX)$.
By Caratheodory theorem, it holds that $\lvert\gamma'\rvert\leq d-f+1$ where $f$ is the dimension of the linear subspace orthogonal to $\{h_\delta\}_{\delta\in\gamma'}$.
We conclude the proof by using the same argument as in case~1: since the problem is now restricted to an $f$-dimensional problem (because $x$ is in the subspace orthogonal to $\{h_\delta\}_{\delta\in\gamma'}$), we may find a set $\beta'\subseteq\omega$, with $\lvert\beta'\rvert\leq f-1$, such that $\Opt(\gamma'\cup\beta')=(\lambda^*,x^*)$.
This shows that $\lvert\beta\rvert\leq\lvert\gamma'\rvert+\lvert\beta'\rvert\leq d$; concluding the proof for the second case.
\end{proofbis}
\endgroup

\section{Chance-constrained quasi-linear optimization}\label{sec-chance-constrained}

Let $\Prob$ be a probability measure on $\Delta$.
Suppose that the constraints $\delta_1,\ldots,\delta_N$ are sampled from $\Delta$ according to $\Prob$, and that we solve the problem $\calP(\omega_N)$ where $\omega_N=\{\delta_1,\ldots,\delta_N\}$.
This approach of solving the optimization problem for a few randomly sampled constraints is called \emph{chance-constrained optimization}.
Under certain assumptions, probabilistic guarantees can be obtained on the measure of the set of constraints $\delta\in\Delta$ that are compatible with the optimal solution of $\calP(\omega_N)$.
This is the case, for instance, for a large class of convex optimization problems \citep[see, e.g.,][]{calafiore2010random} and non-convex optimization problems \citep[though with weaker probabilistic guarantees; see, e.g.,][]{campi2018ageneral}.
In the section, we extend the results from chance-constrained convex optimization \citep{calafiore2010random} to chance-constrained quasi-linear problems.

Therefore, we make the following standing assumption on the set $\Delta$ and on its probability measure $\Prob$.
First, let us introduce the notion of \emph{non-degenerate} vector of constraints.

\begin{definition}\label{def-nondegenerate}
\citep[Definition~2.11]{calafiore2010random}
Let $N\in\NNb_*$.
We say that $\omegab_N\coloneqq(\delta_1,\ldots,\delta_N)\in\Delta^N$ is \emph{non-degenerate} if there is a unique set $I\subseteq\{1,\ldots,N\}$ such that $\{\delta_i\}_{i\in I}$ is an essential set for $\calP(\omega_N)$.
\end{definition}

\begin{assumption}\label{ass-nondegenerate}
\citep[Assumption~2]{calafiore2010random}
For every $N\in\NNb_*$, the vector $\omegab_N\in\Delta^N$ is non-degenerate with probability one.
\end{assumption}

For any vector of constraints $\omegab_N\in\Delta^N$, we define the \emph{violating probability} associated to $\omegab_N$:
\[
V(\omegab_N) = \Prob(\{\delta\in\Delta : \Cost(\omega_N\cup\{\delta\})>\Cost(\omega_N)\}).
\]
We are now able to present the extension of \citet[Theorem~3.3]{calafiore2010random} to chance-constrained quasi-linear programs.
Therefore, let $\zeta\in\NNb_*$ be an upper bound on the cardinality of any essential set of $\calP(\omega)$, with finite $\omega\subseteq\Delta$.
From Theorem~\ref{thm-essential-set-cardinality}, it holds that $\zeta\leq d$.

\begin{theorem}\label{thm-chance-constrained-quasi-linear}
Consider the sampled quasi-linear optimization problem \eqref{eq-optim-sampled}, and let $V(\omegab_N)$ and $\zeta$ be as above.
Let Assumption~\ref{ass-nondegenerate} hold.
Let $N\in\NNb$, $N\geq\zeta$, and let $\varepsilon\in(0,1)$.
Then,%
\[
\Prob^N(\{\omegab_N\in\Delta^N : V(\omegab_N)>\varepsilon\})\leq\Phi(\varepsilon,\zeta-1,N),
\]
where $\Phi(\cdot,\zeta-1,N)$ is the regularized incomplete beta function%
\footnote{See, e.g., \citet[Definition~2]{kenanian2019data}.}%
.
\end{theorem}

\begin{proof}
\citep[Adapted from][]{calafiore2010random}
Fix $N\in\NNb$, $N\geq\zeta$.
By Assumption~\ref{ass-nondegenerate}, we may assume without loss of generality that $\omegab_N$ is non-degenerate for all $\omegab_N\in\Delta^N$.
Hence, for each $\omegab_N\coloneqq(\delta_1,\ldots,\delta_N)\in\Delta^N$, we let $J(\omegab_N)$ be the unique set $I\in\{1,\ldots,N\}$ such that $\{\delta_i\}_{i\in I}$ is an essential set for $\calP(\omega_N)$.
Label the elements of $\Delta$ with labels belonging to a totally order set.%
\footnote{This approach, from \citet{calafiore2010random}, requires the \emph{axiom of choice} when $\Delta$ is a general set.
However, it is not needed for instance if $\Delta\subseteq\Re^n$, as it is the case in our application (see Section~\ref{sec-JSR-black-box}).}
Let $J^*(\omegab_N)$ be a completion of $J(\omegab_N)$ with the $\zeta-\lvert J(\omegab_N)\rvert$ elements of $\{1,\ldots,N\}\setminus J(\omegab_N)$ such that $\{\delta_i\}_{i\in J^*(\omegab_N)\setminus J(\omegab_N)}$ have the largest labels among the elements of $\omega_N$.
From Assumption~\ref{ass-nondegenerate}, it follows that $J^*(\omegab_N)$ is well defined with probability one; hence, in the following, we will assume without loss of generality that $J^*(\omegab_N)$ is well defined for all $\omegab_N\in\Delta^N$.

Let $\{I_1,\ldots,I_M\}$ be the set of all subsets of $\{1,\ldots,N\}$ with $\zeta$ elements; in particular, $M=N!/(\zeta!(N-\zeta)!)\triangleq C(N,\zeta)$.
For each $i=1,\ldots,M$, let $S_i=\{\omegab_N\in\Delta^N:J^*(\omegab_N)=I_i\}$.
The sets $\{S_i\}_{1\leq i\leq M}$ are disjoint and their union is equal to $\Delta^N$.
Moreover, by the symmetry of their definition, they have the same probability; hence $\Prob(S_i)=1/C(N,\zeta)$.

Now, for each $\omegab_\zeta\in\Delta^\zeta$, we let $V^*(\omegab_\zeta)$ be the \emph{violating probability} of $\omegab_\zeta$ with respect to \eqref{eq-optim-sampled} and the labelling of the constraints: that is, $V^*(\omegab_\zeta)=\Prob(\{\delta\in\Delta:J^*(\omegab_\zeta\Vert\delta)\neq\{1,\ldots,\zeta\})$.
From the uniqueness of the optimal solution of the problems $\calP(\omega)$, $\omega\subseteq\Delta$, it follows that for every $L\in\NNb_*$, $\omegab_L\in\Delta^L$ and $\delta,\eta\in\Delta$, if $J^*(\omegab_L\Vert\delta)=J^*(\omegab_L\Vert\eta)=J^*(\omegab_L)$, then $J^*((\omegab_L\Vert\delta)\Vert\eta)=J^*(\omegab_L)$.%
\footnote{See, e.g., \citet[\S2.1]{calafiore2010random} for details.}
It follows that, for any $v\in[0,1]$,
\[
\Prob[S_i\mid V^*(\omegab_{N,i})=v]=(1-v)^{N-\zeta},\quad \forall\,N\geq\zeta,\;i=1,\ldots,C(N,\zeta),
\]
where $\omegab_{N,i}$ is the restriction of $\omegab_N$ to the indices in $I_i$: $\omegab_{N,i}=(\delta_i)_{i\in I_i}$.
Hence, we get that
\begin{equation}\label{eq-Hausdorff}
\Prob(S_i)=\int_0^1 (1-v)^{N-\zeta} \, \diff F_i(v)=1/C(N,\zeta),\quad \forall\,N\geq\zeta,\;i=1,\ldots,C(N,\zeta),
\end{equation}
where $F_i(v)=\Prob^N(\{\omegab_N\in\Delta^N:V^*(\omegab_{N,i})\leq v\})$.
Equation~\eqref{eq-Hausdorff} describes a \emph{Hausdorff moment problem}; it is shown in \citet[p.~3436]{calafiore2010random} that \eqref{eq-Hausdorff} implies that $F_i(v)=v^\zeta$.

Finally, for each $i=1,\ldots,C(N,\zeta)$, we let $B_i=\{\omegab_N\in\Delta^N:V^*(\omegab_{N,i})>\varepsilon\}$.
Using the expression of $F_i$, it can be shown%
\footnote{See, e.g., \citet[Theorem~3.3]{calafiore2010random}.}
that $\Prob^N(B_i)=\Phi(\varepsilon,\zeta-1,N)/C(N,\zeta)$.
By symmetry, we get that $\Prob^N(\bigcup_{1\leq i\leq M}B_i)=\Phi(\varepsilon,\zeta-1,N)$.
Since $\{\omegab_N\in\Delta^N : V(\omegab_N)>\varepsilon\}\subseteq\bigcup_{1\leq i\leq M}B_i$, we obtain the desired result.
\end{proof}

\section{Application to data-driven stability analysis of switched linear systems}\label{sec-JSR-black-box}

Let $\calA=\{A_1,\ldots,A_m\}$ be a fixed set of matrices in $\Re^\nn$, and let $\SSb$ be the unit sphere (boundary of the unit Euclidean ball) in $\Re^n$.
Let $\Delta=\calA\times\SSb$, and let $\Prob$ be the \emph{uniform distribution} on $\Delta$.%
\footnote{I.e., $\Prob=\Prob_1\otimes\Prob_2$ where $\Prob_1$ and $\Prob_2$ are the uniform distributions on $\calA$ and $\SSb$ respectively.}
For a finite set $\omega\subseteq\Delta$, we consider the following sampled quasi-linear optimization problem:
\begin{equation}\label{eq-optim-jsr}
\begin{array}{r@{}l}
\displaystyle\ProbJSR(\omega):\quad \min\limits_{\substack{P=P^\top\in\Re^\nn,\\\gamma\geq0}} \; (\gamma,\lVert P\rVert_F^2) \quad\text{s.t.}\quad& P\in\calX\coloneqq\{P:P\succeq I \:\wedge\: \lVert P\rVert_F\leq C\}, \\[-10pt]
&\displaystyle (Ax)^\top P(Ax) \leq \gamma^2 x^\top Px, \quad \forall\,(A,x)\in\omega,
\end{array}
\end{equation}
for some fixed parameter $C\geq n$.
Note that $\ProbJSR(\omega)$ is a sampled, data-driven version of the classical quadratic Lyapunov framework for the approximation of the \emph{Joint Spectral Radius} (JSR) of the \emph{switched linear system} defined by $\calA$; see, e.g., \citet[Theorem~2.11]{jungers2009thejoint}.
The JSR is a ubiquituous measure of stability of switched linear systems \citep{blondel2005computationally,parrilo2008approximation,jungers2017acharacterization}; it also appears in other areas of hybrid system control, like wireless networked control \citep{berger2020worstcase}.

In order to apply the results from Section~\ref{sec-chance-constrained} on $\ProbJSR(\omega)$, we make the following assumption on the matrices in $\calA$.
First, let us introduce the notion of \emph{Barabanov} matrix.

\begin{definition}\label{def-barabanov}
A matrix $A\in\Re^\nn$ is said to be \emph{Barabanov} if there exists a symmetric matrix $P\succ0$ and $\gamma\geq0$ such that $A^\top PA=\gamma^2P$.
\end{definition}

\begin{assumption}\label{ass-no-barabanov}
There is no Barabanov matrix in $\calA$.
\end{assumption}

We claim that Assumption~\ref{ass-no-barabanov} is not restrictive in most of the practical situations.
To motivate this claim, we provide an equivalent characterization of Barabanov matrices in the proposition below, whose proof can be found in Appendix~\ref{sec-proof-pro-Barabanov}.
For further work, we plan to investigate the possibility to relax or remove this technical assumption.

\begin{proposition}\label{pro-Barabanov}
A matrix $A\in\Re^\nn$ is Barabanov if and only if it is diagonalizable and all its eigenvalues have the same modulus.
\end{proposition}

We now show that Assumption~\ref{ass-no-barabanov} ensures that Assumption~\ref{ass-nondegenerate} holds for \eqref{eq-optim-jsr}.

\begin{proposition}\label{pro-QLP-index-basis}
Consider the sampled problem \eqref{eq-optim-jsr}.
Let Assumption~\ref{ass-no-barabanov} hold.
Then, for every $N\in\NNb_*$, the vector $\omegab_N\in\Delta^N$ is non-degenerate with probability one.
\end{proposition}

We will need the following lemma.

\begin{lemma}\label{lem-zero-set-measure}
Let $P(x_1,\ldots,x_n)$ be a nonzero polynomial on $\Re^n$.
The zero set of $P$, i.e., the set of points $x\in\Re^n$ such that $P(x)=0$, has zero Lebesgue measure.
\end{lemma}

We skip the proof of this well-known fact \citep[see, e.g.,][Problem~2.15]{teschltopics}.

\vspace{\topsep}

\begingroup
\renewcommand{\proofname}{Proof of Proposition~\ref{pro-QLP-index-basis}}
\begin{proofbis}
Let $1\leq i\leq N-1$.
Let us look at the probability that $\beta\coloneqq\{\delta_1,\ldots,\delta_i\}$ is an essential set for $\ProbJSR(\omega)$ and that $\delta_N$ is in another essential set.
This probability is smaller than or equal to the probability that $\beta$ is an essential set for $\ProbJSR(\omega)$ and that $(Ax)^\top P(Ax)=\gamma^2x^\top Px$, where $(\gamma,P)=\OptJSR(\beta)$ and $\delta_N=(A,x)$.

Assume that the above probability is nonzero.
Then, since $\calA$ is finite, that there is $A\in\calA$ such that $(Ax)^\top P(Ax)=\gamma^2x^\top Px$ for all $x$ in a set $S\subseteq\SSb$ with nonzero measure.
Thus, by Lemma~\ref{lem-zero-set-measure}, it holds that $(Ax)^\top P(Ax)=\gamma^2x^\top Px$ for all $x\in\SSb$.
This contradicts the assumption that there is no Barabanov matrix in $\calA$.
Hence, the probability that $\beta$ is a basis for $\ProbJSR(\omega)$ and that $\delta_N$ is in another basis is zero.
Since $\beta$ and $\delta_N$ were arbitrary, this concludes the proof.
\end{proofbis}
\endgroup

Theorem~\ref{thm-chance-constrained-quasi-linear} can thus be applied to $\ProbJSR(\omega)$.

\begin{corollary}\label{cor-chance-constrained-JSR}
Consider the sampled problem \eqref{eq-optim-jsr}.
Let Assumption~\ref{ass-no-barabanov} hold.
Let $N\in\NNb$, $N\geq d\coloneqq\frac{n(n+1)}2$, and let $\varepsilon\in(0,1)$.
Then,
\begin{equation}\label{eq-chance-constrained-inequality}
\Prob^N(\{\omegab_N\in\Delta^N : \ViolJSR(\omegab_N)>\varepsilon\})\leq\Phi(\varepsilon,d-1,N),
\end{equation}
where $\ViolJSR(\omegab_N) = \Prob(\{\delta\in\Delta : \CostJSR(\omega_N\cup\{\delta\})>\CostJSR(\omega_N)\})$.
\end{corollary}

\begin{remark}\label{rem-improvement-chance-constrained-jsr}
We note the improvement of the right-hand side term of \eqref{eq-chance-constrained-inequality}, compared to \citet[Theorem~10]{kenanian2019data}; this term becomes $\Phi(\varepsilon,d-1,N)$ instead of $\Phi(\varepsilon,d,N)$ in \citet{kenanian2019data}.
This is due to the improvement of the bound on the cardinality of essential sets of quasi-linear problems; see Theorem~\ref{thm-essential-set-cardinality}.
\end{remark}

From Corollary~\ref{cor-chance-constrained-JSR}, we deduce the following probabilistic guarantee on the upper bound on the JSR of the switched linear system given by $\calA$, that we can get from the solution of the sampled problem $\ProbJSR(\omega)$.
The derivation of this result follows the same lines as in \citet[Theorems~14 and~15]{kenanian2019data}, so that the details are omitted here.

\begin{corollary}\label{cor-jsr-proba-bound}
Consider the sampled problem \eqref{eq-optim-jsr}.
Let Assumption~\ref{ass-no-barabanov} hold.
Let $N\in\NNb$, $N\geq d\coloneqq\frac{n(n+1)}2$, and let $\varepsilon\in(0,1)$.
Then, for all $\omegab_N\in\Delta^N$, except possibly those $\omegab_N$ in a subset $\Omega\subseteq\Delta^N$ with measure $\Prob^N(\Omega)\leq\Phi(\varepsilon,d-1,N)$, it holds that
\[
\rho(\calA) \leq \gamma^*\Big/\sqrt{1-I^{-1}\Big(\frac{\varepsilon\kappa(P^*)}m;\frac{d-1}2;\frac12\Big)}\:,
\]
where $(\gamma^*,P^*)=\OptJSR(\omega_N)$, $\kappa(P)=\sqrt{\frac{\det(P)}{\lambdamin(P)^n}}$, $I^{-1}$ is the \emph{inversed regularized incomplete beta function}%
\footnote{See, e.g., \citet[Definition~2]{kenanian2019data}.}
and $\rho(\calA)$ is the JSR of the switched linear system defined by $\calA$.
\end{corollary}

\section{Numerical experiments: consensus of hidden network}\label{sec-numerical-experiments}

We consider the problem of consensus in a switching and hidden network.
The interaction between the nodes in the network over time can be modeled as a switched linear dynamical system:
\[
x(t+1) = A_{\sigma(t)} x(t),\quad x(t)\in\mathbb{R}^n,\quad A_{\sigma(t)} \in \mathcal{A} := \{A_1,\ldots, A_m\}\subseteq\Re^\nn,
\]
where $x(t)$ is the state vector ($n$ is the number of nodes) at time $t$ and $A_{\sigma(t)}$ is the interaction matrix at time $t$, with $A_i$ being \emph{unknown} row-stochastic matrices, i.e., $A_i\uh = \uh$, $i=1,\ldots,m$, where $\uh$ is the all-one vector in $\Re^n$.
The goal is to verify that $x(t) = A_{\sigma(t-1)} \cdots A_{\sigma(1)} A_{\sigma(0)} x(0)$ converges to $c\uh$ for some $c$ as $t\to\infty$.
As shown by \citet{jadbabaie2003coordination}, this question boils down to the computation of the JSR of $\calA'\coloneqq\{A'_1,\ldots, A'_m\}\subseteq\Re^{n-1\times n-1}$ where $A'_i=BA_i^{}B^\top$, for $i = 1,\ldots,m$, and $B\in\Re^{n-1\times n}$ is a fixed orthogonal matrix ($BB^\top=I_{n-1}$) with kernel spanned by $\uh$.
In our experiment, we consider a network of $8$ nodes, switching among $3$ modes, as shown in Figure \ref{fig-network}.
The possible networks are not known, and only the state of the different agents is available. Hence, we use the data-driven framework in Section~\ref{sec-JSR-black-box} to estimate the JSR of $\calA'$.

\begin{figure}[h]
\newcommand{\factor}{0.25}
\centering
\begin{tabular}{c@{\qquad}c@{\qquad}c}
\includegraphics[width=\factor\textwidth]{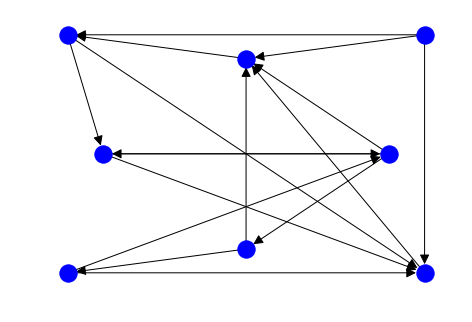} &
\includegraphics[width=\factor\textwidth]{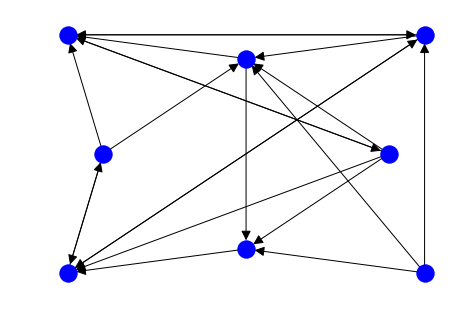} &
\includegraphics[width=\factor\textwidth]{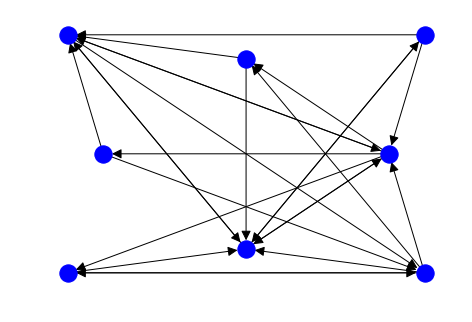}
\end{tabular}
\caption{Example of switching network with $3$ modes.}
\label{fig-network}
\end{figure}

First, we sample a data set of $N$ pairs: $(x_i,y_i)$, with $i=1,\ldots,N$, where $x_i$ is sampled uniformly at random on $\SSb$ and $y_i = A_{\sigma_i} x_i$, with $\sigma_i$ sampled uniformly at random in $\{1,\ldots, m\}$.
This data set is projected onto $\mathbb{R}^{n-1}$ as follows: $(x_i,y_i)\mapsto(x'_i,y'_i)$ where $x'_i=Bx_i$ and $y'_i=By_i$ and $B$ is as above.%
\footnote{The orthogonality of $B$ is important to ensure that $x'_i/\lVert x'_i\rVert$ is distributed uniformly on $\SSb$.} We then solve the problem in Section~\ref{sec-JSR-black-box} with the projected data set.
We fix the confidence level at $\beta=0.05$.
The probabilistic upper bound on the JSR obtained from Corollary~\ref{cor-jsr-proba-bound} is shown in Figure \ref{fig-jsr} for different sizes of the sample set.
For a comparison, the bound of \citet{kenanian2019data} is also given.
While both bounds converge when the number of samples increases, the bound in this paper requires fewer samples to deduce convergence of the system to consensus, with the same confidence level.

\begin{figure}[h]
\centering
\includegraphics[width=0.9\textwidth]{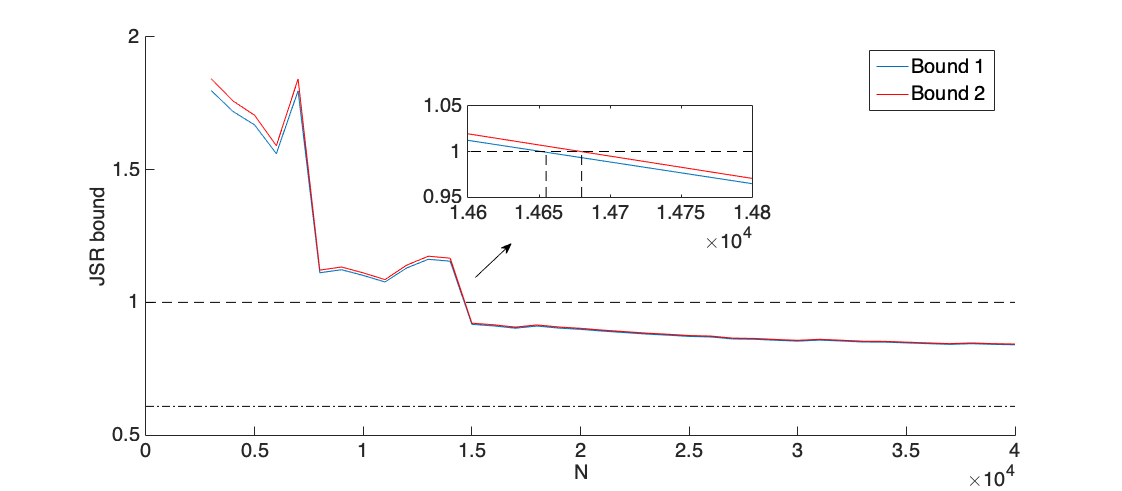}
\caption{Data-driven upper bounds on the JSR for different sizes of the sample set.
Bound 1 refers to the bound of this paper, Bound 2 refers to the bound of \citet{kenanian2019data}, and the dash-dotted line is the bound computed from the white-box model, using the JSR toolbox \citep{jungers2009thejoint}.}
\label{fig-jsr}
\end{figure}

\section{Conclusions}

In this work, we generalized the theory of chance-constrained optimization to quasi-convex problems, and pushed further the effort initiated in \citet{kenanian2019data}, demonstrating its use for data-driven stability analysis of complex systems.
More precisely, we introduced the class of quasi-linear optimization problems, which is a subclass of quasi-convex problems.
We particularized and improved some classical results of quasi-convex programming to this class.
This allowed us to extend the results of chance-constrained convex optimization to quasi-linear optimization problems.
Thriving on this, we provided a proof of concept that quasi-linear problems are useful for data-driven control applications.
In particular, we applied our framework to the problem of JSR approximation of black-box switched linear systems, introduced in \citet{kenanian2019data}.

For future work, we plan to investigate other applications of chance-constrained quasi-linear optimization for data-driven control.
For instance, we believe that by replacing the conic constraints with their sampled counterpart, one could transform many optimization problems in control theory into quasi-linear programs, and then use chance-constrained optimization to bridge the gap between the original and the sampled formulations.
We also plan to investigate the possibility of relaxing or removing the assumption that there are no Barabanov matrices involved in the switched linear system.
Finally, we plan to provide other approaches for the data-driven stability analysis of switched linear systems based on chance-constrained quasi-linear optimization (e.g., thriving on sum-of-square optimization or path-complete Lyapunov frameworks).

\appendix

\section{Proof of Proposition~\ref{pro-Barabanov}}\label{sec-proof-pro-Barabanov}

First, we prove the \emph{if} direction:
Assume that $A$ is diagonalizable and all its eigenvalues have the same modulus.
Then, there is $T\in\Re^\nn$ invertible such that $A=T^{-1}DT$, where $D\in\Re^\nn$ is block-diagonal with diagonal blocks of size $1$ or $2$, corresponding to eigenvalues with the same modulus.
Denote this common modulus by $\gamma$.
Now, let $P=T^\top T$, which is positive definite.
We verify that $A^\top P A = T^\top D^\top DT = \gamma^2 T^\top T=\gamma^2P$.
Hence, $A$ is Barabanov.

Now, we show the \emph{only if} direction:
Assume that $A^\top PA=\gamma^2P$ for some $P\succ0$ and $\gamma\geq0$.
Let $P=L^\top L$ be a Cholesky factorization of $P$.
It follows that $B^\top B=\gamma^2 I$, where $B=LAL^{-1}$.
If $\gamma=0$, this implies that $B=0$ and thus $A=0$, proving the \emph{only if} direction when $\gamma=0$.
If $\gamma>0$, this implies that $B/\gamma$ is a unitary matrix.
It follows that $B$ is diagonalizable and all its eigenvalues have modulus $\gamma$.
Now, since $A$ is similar to $B$, the same holds for $A$, proving the \emph{only if} direction when $\gamma>0$.\jmlrQED

\bibliography{myrefs}

\end{document}